\documentclass{siamart1116}

\newtheorem*{remark}{Remark}
\numberwithin{equation}{section}
\numberwithin{table}{section}
\numberwithin{figure}{section}

\usepackage{mathtools}

\usepackage{lipsum}
\usepackage{amsfonts}
\usepackage{graphicx}
\usepackage{epstopdf}
\usepackage{algorithmic}
\usepackage[framed,numbered,autolinebreaks,useliterate]{mcode}
\usepackage{amsmath}
\usepackage{amssymb}
\usepackage{algorithm}
\usepackage{bbm}
\usepackage{color}
\usepackage{tikz}
\usetikzlibrary{shapes,arrows}
\newcommand*\diff{\mathop{}\!\mathrm{d}}

\ifpdf
  \DeclareGraphicsExtensions{.eps,.pdf,.png,.jpg}
\else
  \DeclareGraphicsExtensions{.eps}
\fi

\numberwithin{theorem}{section}

\newcommand{\TheTitle}{Convergence of Milstein Brownian bridge Monte Carlo methods and stable Greeks calculation} 
\newcommand{\TheAuthors}{Thomas Gerstner, Bastian Harrach, Daniel Roth}

\headers{Convergence of Milstein Brownian bridge Monte Carlo methods}{\TheAuthors}

\title{{\TheTitle}\thanks{Submitted to the editors \today .
}}

\author{
  Thomas Gerstner\thanks{Department of Mathematics, Goethe University Frankfurt, Germany (\email{gerstner@math.uni-frankfurt.de},
    \email{harrach@math.uni-frankfurt.de}, \email{roth@math.uni-frankfurt.de}).}
  \and
  Bastian Harrach\footnotemark[2]
  \and
  Daniel Roth\footnotemark[2]
}

\usepackage{amsopn}

\tikzstyle{decision} = [diamond, draw, fill=blue!20, 
    text width=4.5em, text badly centered, node distance=3cm, inner sep=0pt]
\tikzstyle{block} = [rectangle, draw,  
    text width=6em, text centered, rounded corners, minimum height=2.5em]
\tikzstyle{blocks} = [rectangle, draw,  
    text width=5em, text centered, rounded corners, minimum height=2em]
    \tikzstyle{blockss} = [rectangle, draw,  
    text width=2.9em, text centered, rounded corners, minimum height=2em]
\tikzstyle{line} = [draw, -latex']
\tikzstyle{cloud} = [draw, ellipse,fill=red!20, node distance=3cm,
    minimum height=2em]

\ifpdf
\hypersetup{
  pdftitle={\TheTitle},
  pdfauthor={\TheAuthors}
}
\fi

\begin{document}

\maketitle


\begin{abstract}
We consider in this work the pricing and the sensitivity calculation of continuously monitored barrier options. Although standard Monte Carlo algorithms work well for pricing these options, they often do not behave in a stable way with respect to numerical differentiation. To bypass this problem, one would generally either resort to regularized differentiation schemes or derive an algorithm for precise differentiation. While the widespread solution of using a Brownian bridge approach leads to accurate first derivatives, they are not Lipschitz-continuous. This leads to instability with respect to numerical differentiation for second-order Greeks.

To alleviate this problem - i.e. produce Lipschitz-continuous first-order derivatives - and reduce variance, we generalize the idea of one-step survival to general scalar stochastic differential equations. This approach leads to the new one-step survival Brownian bridge approximation, which allows for stable second-order Greeks calculations.

This work proves unbiasedness and variance reduction of our new, one-step survival version, with respect to the classical, Brownian bridge approach. Furthermore, we will present a new convergence result for the Brownian bridge approach using the Milstein scheme under certain conditions. Overall, these properties imply convergence of the new one-step survival Brownian bridge approach. 

To show the new approach's numerical efficiency, we present a respective Monte Carlo pathwise sensitivity estimator for the first-order Greeks and study different approaches to compute second-order Greeks stably. Finally, we develop a one-step survival Brownian bridge multilevel Monte Carlo algorithm to reduce the computational cost in practice.

\end{abstract}

\begin{keywords}
  Monte Carlo, barrier options, pathwise sensitivities, Brownian bridge, one-step survival, second-order Greeks
\end{keywords}

\section{Introduction}
In computational finance, Monte Carlo methods are used extensively in the pricing of financial derivatives and quantitative risk management \cite{glasserman,asmussen2007stochastic}. We consider Monte Carlo pricing schemes for the prices and sensitivities of different types of exotic options with discontinuous payoffs, especially the continuously monitored barrier options. Whether a specific pre-defined barrier condition is fulfilled or not, a continuously monitored barrier option is "knocked-in" or "knocked-out" when the underlying asset crosses this barrier. For an overview of other exotic options, see Zhang \cite{Zhang}, particularly for options with a discontinuous payoff. For an overview of various approaches that aim to price specific types of exotic options through Monte Carlo simulation, we refer to the monograph by Glasserman \cite{glasserman}.

In practice, both discretely \cite{broadie1997continuity}, and continuously monitored barrier options \cite{yang2017pricing} are among the most frequently traded derivatives. Therefore, it is essential to price barrier options under flexible models, irrespective of the monitoring frequency, to describe the observed market option prices. Contributing to that effort, we look at pricing and Greeks calculation of continuously monitored barrier options. For an overview of discretely monitored barrier options, we refer to \cite{gerstner2018monte}. Even though some analytical pricing formulas exist for some basic models, see, e.g.,  \cite{merton1973theory,briys1998options}, it is well known that the classical Black-Scholes model lacks the necessary flexibility to fit the observed market data, see, e.g., \cite{gatheral2011volatility}. For the study of more complex stochastic models, Monte Carlo simulation remains the preferred approach for pricing. 

Besides pricing, financial institutions need to evaluate the sensitivities of their portfolios due to regulations. 
Using finite-difference approximations to compute option sensitivities is simple but can also be hard to control since it generally leads to unstable results. Even the smallest numerical errors may have arbitrarily large effects on the finite difference approximation. This property, known as \textit{ill-posedness}, cf., Engl, Hanke and Neubauer \cite{engl1996regularization}, requires further studies.
The challenges for barrier options are usually handled in the following ways: payoff-smoothing combined with finite differences or the Likelihood Ratio Method. For discretely observed barrier options, Alm et al. \cite{alm} determined that the first approach is computationally more efficient. Furthermore, the authors show that a Monte Carlo pricing algorithm that uses a Lipschitz-continuous payoff allows stable differentiation by simple finite differences. We will see that the Brownian bridge approach's lack of Lipschitz-continuity leads to second-order Greeks' instability with respect to numerical differentiation. To overcome this problem, we combine the one-step survival strategy \cite{staum} with the Brownian bridge method to obtain the new one-step survival Brownian bridge approach that allows a stable second-order Greek computation through finite differences. 

The central part of this work will be to show the one-step survival Brownian bridge approaches' unbiasedness and variance reduction. Specifically, in this work, we consider unbiasedness to be understood with respect to the Brownian bridge approach or its derivatives. The unbiasedness and certain convergence properties of the Brownian bridge approach, which we will discuss later, will lead to the new approach's overall convergence. Moreover, we will present the first partial derivatives, show their unbiasedness and present the respective pathwise sensitivity Monte Carlo estimator. In \cite{burgosdis}, Burgos proves that, under certain assumptions, the first derivatives of the Brownian bridge approach satisfy weak convergence. Again, together with the unbiasedness, this convergence property will guarantee convergence of the new approach's first derivatives. Then, we discuss the one-step survival Brownian bridge approach's main practical advantage: Stable second-order Greeks calculation. Therefore, we discuss second-order pathwise sensitivities and extend the result of Alm et al. \cite{alm} to stable second-order Greeks through finite differences.

Neither the pathwise sensitivities nor the Likelihood Ratio Method can be used universally. Other estimators for price approximations and their Greeks using the Malliavin calculus tools have been developed in \cite{gobet2003computation,fournie1999applications,benhamou2000application}. However, in general, Malliavin calculus, which is beyond this work's scope, leads to random variables whose simulations are costly in computational time. In \cite{gobet2004revisiting}, the authors, rather than using Malliavin calculus techniques, successfully take advantage of the models' Markovian structure. Their essential idea consists of using suitable martingals to derive integration by parts formula, see, e.g., \cite{thalmaier1998gradient, bismut1984large}, however, this is restricted to the Delta for barrier options and bounded coefficient terms. For further developments on the combinations of pathwise sensitivities and the Likelihood Ratio Method, see, e.g., \cite{giles2009vibrato}. For further studies on the law's properties associated with the first hitting time using bounded coefficient assumptions and a parametric method, \cite{andersson2017unbiased,bally2015probabilistic}, we refer to Frikha Kohatsu-Higa and Li \cite{frikha2016first}.

In this work, we are interested in the expected value $\mathbb{E}[P]$ of a quantity that is a functional $P$ of the solution of a stochastic differential equation (SDE) with a general drift and volatility term. It is well known that under certain conditions, see, e.g., \cite{kloeden2012numerical}, one obtains convergence for the expectation of a Lipschitz-continuous payoff if it only depends on the solutions' time of maturity. In particular, for pricing an option, we are interested in the (weak) convergence
\begin{align}
 \mathbb{E}[P-\widehat{P}]   \leq C h^\alpha,\label{alpha}
\end{align}
with a constant $C>0$ and $\alpha>0$ for the approximated expected payoff $\widehat{P}$ evaluated on a discretization of the SDE using a certain step width $h$. Using the appropriate conditions on the drift and volatility terms, we know that the Euler-Maruyama or Milstein schemes typically converge for specific payoffs. E.g., they converge for $\mathcal{C}_P^{4}$ with $\alpha=1$, for time-independent payoffs satisfying $P \in \mathcal{C}_P^{4}$ (the space of functions such that all partial derivatives up to and including $4$ exist, are continuous and have polynomial growth), see, e.g., \cite{kloeden2012numerical}.\\
We know from Asmussen, Glynn, and Pitman \cite{asmussen1995discretization} that the convergence order is bounded by $\alpha=1/2$ for any path-dependent payoff using the maximum (minimum) of a discrete approximation as an approximation of the maximum (minimum), as one does by default for the barrier options. To recover a convergence of order one, the most frequently used approach is the Brownian bridge interpolation, which samples if the maximum exceeds the barrier between two steps. 
 
From Gobet \cite{gobet2000weak}, we know that this approach recovers $\alpha=1$ using the Euler-Maruyama scheme while assuming stronger conditions on the drift and volatility, i.e., $C_b^\infty$, and on the support or regularity conditions on the payoff, i.e., the support strictly excluded from, or vanishing at the barrier. Gobet explains how the drift and volatility conditions can be weakened depending on the payoff structure, i.e., to $C_b^5$ on the volatility for a vanishing payoff or to $C_b^7$ for a growth bounded payoff. 

In the recent work on the Multilevel Monte Carlo analysis, Giles, Debrabant, and Roessler \cite{giles2013numerical} show there is weak order $\alpha=1$ for Lipschitz-continuous and Asian options using the Milstein scheme for weaker assumptions. Furthermore, for barrier options using the Brownian bridge approach with the Milstein scheme, they prove the multilevel estimator's convergence. They point out that their work could be modified to show that the Brownian bridge approach using the Milstein scheme satisfies $\alpha=1-\delta$, for any $\delta>0$, under these weaker assumptions. For the sake of completeness, we will prove this property in the appendix. We will formulate the theorem on unbiasedness and variance reduction quite generally, so that it can be applied independently of specific model assumptions or respective weak convergence assumptions. We will use the Milstein scheme in the main theorem due to its beneficial properties for further computational savings. Nevertheless, the application to the less complex Euler-Maruyama scheme is straightforward, as we will present in a lemma. 

Monte Carlo methods can be computationally expensive, as in stochastic differential equations, particularly since the cost of generating the individual stochastic samples is very high. It is well established that the computational complexity (cost) to achieve an error $\epsilon$ is of order $\mathcal{O}(\epsilon^{-3})$, see, e.g., \cite{duffie1995efficient}, provided that the stochastic differential equations satisfy certain conditions \cite{kloeden2012numerical,bally1996law,talay1990expansion}. Giles \cite{giles2008multilevel,giles2008improved} shows that multigrid ideas can be used to reduce the computational complexity to $\mathcal{O}(\epsilon^{-2})$ using the multilevel Monte Carlo and the Milstein scheme. The Multilevel Monte Carlo method got various generalizations and extensions, see, e.g.,  \cite{giles2015multilevel}, for an overview. Giles, Debrabant, and Roessler \cite{giles2013numerical} showed that the Brownian bridge approach for barrier options satisfies the necessary convergence properties leading to an efficient multilevel Monte Carlo method. Unfortunately, this approach cannot be straightforwardly applied to the one-step survival Brownian bridge estimator since the coarse path modification used would lead to biased survival probabilities. Nevertheless, we present a modification to the one-step survival Brownian bridge approach and show its numerical efficiency. For further studies on the Brownian bridge multilevel Monte Carlo approach, we refer to \cite{burgos,burgosdis}. 

The structure of this work is as follows. In section 2, we present the main result of the one-step survival Brownian bridge approximation. Then, we study first and second-order Greeks, including the pathwise sensitivity approximation for the one-step survival Brownian bridge approach and the stability result for second-order finite differences. In section 3, we present the multilevel Monte Carlo algorithm for the new approach. Numerical results for the variance reduction, Greeks' stability, and the multilevel algorithm's efficiency are provided in section 4. Section 5 contains some concluding remarks. In the appendix, we present the convergence result.

\section{One-step survival Brownian bridge Monte Carlo estimator for continuously monitored barrier options}

We will focus on barrier options, which only depend on one underlying asset. We are interested in the expected value of a payoff $P$ that is a functional of the asset price. In particular, we suppose that we have the following model for the asset price.
\begin{definition}\label{SDE}
The underlying asset price $S(t)$ is a continuous-time stochastic process whose evolution SDE is of the generic form
\begin{align}
dS(t)=\mu (S,t)dt+\sigma(S,t)dW(t),\label{sde}
\end{align}
on the time interval $t\in [t_0,T]$, with initial value $S_0$, drift $\mu$, volatility $\sigma$ and the Brownian motion $W$.
\end{definition}

Now, we will introduce continuously monitored barrier options.

\begin{definition}
The payoff of a continuously observed up-and-out barrier call option is given by
\begin{align}
P(S):=
   \begin{cases}
     \left(S(T)-K\right)^+=:q(S(T)) & \max\limits_{t\in[t_0,T]}S(t) \le B \label{payoffcont} \\
     0 & \text{otherwise, }  
   \end{cases}
\end{align}
with barrier value $B$, strike price $K$, time of maturity $T$ and current time $t_0$.
\end{definition}

As stated above, we are interested in such an instrument's expected value, which is defined as follows, where for the sake of simplicity, we set $r=0$.

\begin{definition}\label{payoffcontiniuousupout2}
The present value of an option with payoff \eqref{payoffcont} is given by the  discounted expected payoff
\begin{align*}
PV_{t_0}=\mathbb{E}[P(S)],
\end{align*}
at the current time $t_0$ and at the time of maturity $T$.
\end{definition}

Classic weak convergence theory, see, e.g., \cite{kloeden2012numerical}, shows, that the Milstein scheme
\begin{align}
\begin{split}\widehat{S}_{n+1}=\widehat{S}_n&+\mu(\widehat{S}_n,t_n)h+ \sigma(\widehat{S}_n,t_n) \sqrt{h}\Delta Z_n \\&+ \frac{1}{2}\sigma(\widehat{S}_n,t_n) \sigma'(\widehat{S}_n,t_n)\left( (\sqrt{h} \Delta Z_n)^2-h\right),\end{split} \label{milstein}
\end{align}
where $\sigma'$ denotes the derivative of $\sigma(S,t)$ with respect to $S$, with $n=0,\dots,N-1$ discretization steps, $\Delta Z_n \sim N(0,1)$, $\widehat{S}_0=S_0$, $h=T/N$, $t_n$ the time $t$ at step $n$, converges in the sense of \eqref{alpha}, with $\alpha=1$ for the Lipschitz-continuous payoff functions (only depending on the time of maturity), see, e.g., \cite{glasserman}. The same result holds for the Euler-Maruyama scheme, which is defined similarly but without the last summand, i.e., if the derivative of the volatility vanishes both schemes are equal.
From Asmussen, Glynn, and Pitman \cite{asmussen1995discretization}, we know that for any path-dependent payoff using the maximum (minimum) of a discrete approximation as an approximation of the maximum (minimum), as one would do by default for barrier options, the convergence order is bounded by $\alpha=1/2$.
 
The difficulty of recovering the convergence order for the barrier options can be circumvented by sampling if the maximum exceeds the barrier between two discretization steps, instead of sampling the maximum itself, see, e.g., Glasserman \cite{glasserman} for a derivation. We use
\begin{align}
\widehat{p}_n=\exp \left( \frac{-2(B-\widehat{S}_n)^+(B-\widehat{S}_{n+1})^+}{\sigma(\widehat{S}_{n},t_n)^2 h}  \right)\label{brownianp},
\end{align}
with $n=0,\dots,N-1$, which is the conditional probability of crossing the barrier between two steps. Then, the Brownian bridge approximation of the payoff \eqref{payoffcont} is defined by
\begin{align}
\begin{split}
\widehat{P}\left(\widehat{S}_N,\widehat{p}_0,\dots,\widehat{p}_{N-1}\right)= q(\widehat{S}_{N}) \prod\limits_{n=0}^{N-1} \left( 1- \widehat{p}_n\right) ,\end{split}\label{BBint}
\end{align} 
with $q$ from \eqref{payoffcont}. 
By sampling a sequence of possible realizations $(\widehat {s}_{1,N},\dots ,\widehat{s}_{M,N})$, $m=1,\dots ,M$, of the random variables $(\widehat{S}_1,\dots,\widehat{S}_N)$, we obtain the one-step survival Brownian bridge Monte Carlo estimator for $PV_{t_0}$, see e.g. \cite{glasserman}. 
\begin{corollary}\label{mc1}
The Brownian bridge Monte Carlo estimator for the present value of a up-and-out barrier option is given by the average
\begin{align*}
\overline{P_{M}}:=\frac{1}{M}\sum\limits_{m=1}^{M}\left[ q(\widehat {s}_{N,m})\prod\limits_{n=0}^{N-1} \left( 1- \widehat{p}_{n,m}\right)\right].
\end{align*}
\end{corollary}
\Cref{test1} presents a procedure for an estimator of \cref{mc1} while using the Milstein scheme. The algorithm can be easily modified to use the Euler-Maruyama scheme by omitting the last summand of the Milstein scheme of line 7. Fixing the discretization scheme still leaves open some flexibility in choosing the process to apply it, e.g., for a geometric Brownian motion, the transformation to $X(t)=\log S(t)$ would be exact and lead to a constant volatility coefficient.  
\begin{algorithm}\label{algorithm1}
\caption{The Brownian bridge Monte Carlo estimator for an up-and-out barrier option}\label{test1}
\begin{algorithmic}[1]
\STATE Initialize random seed
\FOR {$m=1,\dots,M$} 
\FOR {$n=0:N-1$}
\STATE Sample $z \sim N(0,1)$
\STATE $\widehat{S}_{n+1}=\widehat{S}_n+\mu(\widehat{S}_n,t_n)h+ \sigma(\widehat{S}_n,t_n) \sqrt{h}z$
\STATE \hspace{30mm}$+ \frac{1}{2}\sigma(\widehat{S}_n,t_n) \sigma'(\widehat{S}_n,t_n)\left( (\sqrt{h} z)^2-h\right)$
\STATE $\widehat{p}_n=\exp \left( \frac{-2(B-\widehat{S}_n)^+(B-\widehat{S}_{n+1})^+}{\sigma(\widehat{S}_{n},t_n)^2 h}  \right)$
\ENDFOR
\STATE $\widehat{P}_m= q(\widehat{S}_N)  \prod_{n=0}^{N-1}(1-\widehat{p}_n) $
\ENDFOR
\STATE \textbf{return} $PV_{t_0}:=\frac{1}{M}\sum_{m=1}^M \widehat{P}_m$
\end{algorithmic}
\end{algorithm}

\subsection{One-step survival Brownian bridge approximation}
We now combine the Brownian bridge approximation with the one-step survival idea of Glasserman and Staum \cite{staum} to define the one-step survival Brownian bridge approximation and show that it leads to an unbiased expectation and a variance reduction. Furthermore, we will see that the new approximation will be Lipschitz-continuous since we smooth out the indicator functions of the crossing probabilities \eqref{brownianp}. Now, we will shortly explain the new approach for a simplified case. Then, we present an algorithm and the theorem on unbiasedness and variance reduction.

Consider using only $h=T$ for the Milstein scheme, i.e., $N=1$. For this simplified case, the expectation of the Brownian bridge \eqref{BBint} for given $S_0$ is given by
\begin{align*}
\mathbb{E}[\widehat{P}]= \int\limits_\mathbb{R}\phi(z) \left( 1- \widehat{p}_0\right)q\left(\widehat{S}_1(z)\right)\diff z,
\end{align*}
with the standard normal density $\phi$, the crossing probability
\begin{align*}
\widehat{p}_0&=\exp \left( \frac{-2(B-S_0)^+(B-\widehat{S}_{1})^+}{\sigma(S_0,t_0)^2 h}  \right),
\end{align*}
and the discretization step
\begin{align*}
\widehat{S}_{1}(z)&=S_0+\mu(S_0,t_0)h+ \sigma(S_0,t_0)\sqrt{h} z + \frac{1}{2}\sigma(S_0,t_0) \sigma'(S_0,t_0){h}\left( z^2-1\right).
\end{align*}
The issue of non-Lipschitz-continuous first-derivatives arises since $\widehat{S}_{1}$ could lay above the barrier $B$ and we therefore need the indicator functions in $\widehat{p}_0$. To overcome this problem we use the following considerations: Since we know that $1-\widehat{p}_0$ becomes zero if $\widehat{S}_1$ lays above the barrier we obtain
\begin{align*}
\mathbb{E}[\widehat{P}]=\int\limits_\mathbb{R}\phi(z) \left( 1- \widehat{p}_0\right)q\left(\widehat{S}_1(z)\right)\diff z= \int\limits_{\widehat{S}_1(z)<B}\phi(z) \left( 1- \widehat{p}_0\right)q\left(\widehat{S}_1(z)\right)\diff z.
\end{align*}
Here, $\phi(z)$ is no longer a probability density and so use the normalization
\begin{align*}
\widetilde{p}_0:&=\int_{\widehat{S}_1(z)<B} \phi(z) \diff z
\end{align*}
to obtain
\begin{align}
\mathbb{E}[\widehat{P}]=\widetilde{p}_0\int\limits_{\widehat{S}_1(z)<B}\frac{\phi(z)}{\widetilde{p}_0} \left( 1- \widehat{p}_0\right)q\left(\widehat{S}_1(z)\right)\diff z.\label{pvinsert}
\end{align}
By using this normalized probability density the indicator functions in $\widehat{p}_0$ would not be required anymore, leading to a Lipschitz-continuous first derivative. However, for a Monte Carlo estimator the current representation would need to sample from the normalized density in each discretization step. Therefore, we will transform the integral to the unit cube. This will lead to only uniformly sampled random variables being required for a Monte Carlo estimator. Consider $\sigma'({S}_0,t_0)\neq 0$, we have that $\widehat{S}_{1}(z)<B$ if and only if
\begin{align*}
0>-\frac{B-\widehat{S}_1-\mu({S}_0,t_0)h +\frac{1}{2}\sigma({S}_0,t_0) \sigma'({S}_0,t_0)h}{\sigma({S}_0,t_0) \sqrt{h}}+ z + \frac{1}{2} \sigma'({S}_0,t_0) \sqrt{h}z^2.
\end{align*}
Therefore, we define
\begin{align*}
z<\frac{-1+ \sqrt{1\pm 4\left(\frac{1}{2} \sigma'({S}_0,t_0) \sqrt{h}\right)\left(\frac{B-{S}_0-\mu({S}_0,t_0)h +\frac{1}{2}\sigma({S}_0,t_0) \sigma'({S}_0,t_0)h}{\sigma({S}_0,t_0) \sqrt{h}}\right)}}{  \sigma'({S}_0,t_0) \sqrt{h}}\notag\\:=\Phi^{-1}(\widetilde{p}_0^{\pm})
\end{align*}
with the cumulative standard normal distribution function $\Phi$ and $\widetilde{p}_0:=\widetilde{p}_0^{+}-\widetilde{p}_0^{-}$. Now, we have\begin{align*}
\mathbb{E}[\widehat{P}]=\widetilde{p}_0\int\limits_{\Phi^{-1}(\widetilde{p}_0^{\text{-}})}^{\Phi^{-1}(\widetilde{p}_0^{\text{+}})}\frac{\phi(z)}{\widetilde{p}_0} \left( 1- \widehat{p}_0\right) q\left(\widehat{S}_1(z)\right)\diff z.
\end{align*}
By substituting with $z=\Phi^{-1}(\widetilde{p}_0^{\text{-}}+(\widetilde{p}_0^{\text{+}}-\widetilde{p}_0^{\text{-}}) u)$, we obtain
\begin{align}
\mathbb{E}[\widehat{P}]=  \widetilde{p}_0\int\limits_{0}^{1}\left( 1- \widehat{p}^*_0\right)q\left(\widetilde{S}_1(u)\right)\diff u := \mathbb{E}[\widetilde{P}],
\end{align}
with the modified discretization step
\begin{align*}
\widetilde{S}_{1}(u)=S_0+&\mu(S_0,t_0)h+ \sigma(S_0,t_0)\Phi^{-1}(\widetilde{p}_0^{\text{-}}+u \widetilde{p}_0)\\&+\frac{1}{2}\sigma(S_0,t_0) \sigma'(S_0,t_0)h\left( (\Phi^{-1}(\widetilde{p}_0^{\text{-}}+u \widetilde{p}_0 ))^2-1\right)
\end{align*}
and the modified probability
\begin{align}
\widehat{p}^*_0:=\exp \left( \frac{-2(S_0-B)(\widetilde{S}_{1}-B)}{\sigma(S_0,t_0)^2 h}  \right).\label{p0n1}
\end{align}
For the easier case $\sigma'({S}_0,t_0)=0$, the Milstein scheme simplifies to the Euler scheme and we can transform the integral by using the substitution $z=\Phi^{-1}(\widetilde{p}_0 u)$, with 
\begin{align*}
\widetilde{p}_0&=\Phi \left( \frac{B-{S}_0-\mu({S}_0,t_0)h}{\sigma({S}_0,t_0)\sqrt{h}} \right).
\end{align*}
We know that $\widehat{S}_{1}(z)<B$ if, and only if,
\begin{align*}
0>-\frac{B-{S}_0-\mu({S}_0,t_0)h}{\sigma({S}_0,t_0) \sqrt{h}}+ z.
\end{align*}
Considering general $N$, we can iteratively split the integral expression; see \cref{thmBBOSS} for the profound derivation. 
\begin{corollary}\label{mcoss}
The new one-step survival Brownian bridge Monte Carlo estimator for the expected value is given by
\begin{align}
\overline{P_{M}}:=\frac{1}{M}\sum\limits_{m=1}^{M}\left[ q(s_{N,m}) \prod\limits_{n=0}^{N-1} \left( 1- \widehat{p}^*_{n,m}\right)\prod\limits_{n=0}^{N-1}  \widetilde{p}_{n,m}\right].\label{mc2}
\end{align}
\end{corollary}
\Cref{test2} presents a procedure for an estimator of \cref{mc1} while using the Milstein scheme.

\begin{algorithm}\label{algorithm2}
\caption{The one-step survival Brownian bridge Monte Carlo estimator for an up-and-out barrier option}\label{test2}
\begin{algorithmic}[1]
\STATE Initialize random seed
\FOR {$m=1,\dots,M$} 
\FOR {$n=0:N-1$}
\IF{$\sigma'(\widetilde{S}_{n},t_n)\neq 0$}
\STATE $\widetilde{p}^{+}_n=\Phi \left(\frac{-1+ \sqrt{1+4\left(\frac{1}{2} \sigma'(\widetilde{S}_n,t_n) \right)\left(\frac{B-\widetilde{S}_n-\mu(\widetilde{S}_n,t_n)h +\frac{1}{2}\sigma(\widetilde{S}_n,t_n) \sigma'(\widetilde{S}_n,t_n)h}{\sigma(\widetilde{S}_n,t_n) }\right)}}{ \sigma'(\widetilde{S}_n,t_n) \sqrt{h}}\right)$
\STATE $\widetilde{p}^{-}_n=\Phi \left(\frac{-1- \sqrt{1+4\left(\frac{1}{2} \sigma'(\widetilde{S}_n,t_n) \right)\left(\frac{B-\widetilde{S}_n-\mu(\widetilde{S}_n,t_n)h +\frac{1}{2}\sigma(\widetilde{S}_n,t_n) \sigma'(\widetilde{S}_n,t_n)h}{\sigma(\widetilde{S}_n,t_n) }\right)}}{ \sigma'(\widetilde{S}_n,t_n) \sqrt{h}}\right)$
\STATE $\widetilde{p}_n=\widetilde{p}^{+}_n- \widetilde{p}^{-}_n$
\STATE Sample $u \sim U(0,1)$
\STATE $\widetilde{S}_{n+1}=\widetilde{S}_{n}+\mu(\widetilde{S}_{n},t_n)h+ \sigma(\widetilde{S}_{n},t_n) \sqrt{h}\Phi^{-1}(\widetilde{p}^{-}_n+\widetilde{p}_n u)$
\STATE \hspace{30mm}$\frac{1}{2}\sigma(\widetilde{S}_n,t_n) \sigma'(\widetilde{S}_n,t_n)h\left( (\Phi^{-1}(\widetilde{p}^{-}_n+\widetilde{p}_n u))^2-1\right)$

  \ELSE
    \STATE $\widetilde{p}_n=\Phi \left( \frac{B-\widetilde{S}_n-\mu(\widetilde{S}_n,t_n)h}{\sigma(\widetilde{S}_n,t_n)\sqrt{h}} \right)$
    \STATE Sample $u \sim U(0,1)$
\STATE $\widetilde{S}_{n+1}=\widetilde{S}_{n}+\mu(\widetilde{S}_{n},t_n)h+ \sigma(\widetilde{S}_{n},t_n) \sqrt{h}\Phi^{-1}(\widetilde{p}_n u)$
  \ENDIF
\STATE $\widehat{p}^*_n=\exp \left( \frac{-2(B-\widetilde{S}_n)(B-\widetilde{S}_{n+1})}{\sigma(\widetilde{S}_{n},t_n)^2 h}  \right)$  
\ENDFOR
\STATE $\widetilde{P}_m=q(\widetilde{S}_N) \prod_{n=0}^{N-1}(1-\widehat{p}_n^*)\prod_{n=0}^{N-1}\widetilde{p}_n $
\ENDFOR
\STATE \textbf{return} $PV_{t_0}:=\frac{1}{M}\sum_{m=1}^M \widetilde{P}_m$
\end{algorithmic}
\end{algorithm}

Now, we present the main theorem, proving unbiasedness and a variance reduction property for the new one-step survival Brownian bridge approach.

\begin{theorem}\label{thmBBOSS}
Consider an SDE \eqref{sde} with positive volatility and the modified Milstein scheme:
\begin{align}
\begin{split}\widetilde{S}_{n+1}&=\widetilde{S}_{n}+\mu(\widetilde{S}_{n},t_n)h+ \sigma(\widetilde{S}_{n},t_n) \sqrt{h}\Phi^{-1}(\widetilde{p}^{-}_n+\widetilde{p}_n u_{n})\\
&\quad \quad + \frac{1}{2}\sigma(\widetilde{S}_n,t_n) \sigma'(\widetilde{S}_n,t_n)h\left( (\Phi^{-1}(\widetilde{p}^{-}_n+\widetilde{p}_n u_{n}))^2-1\right),\end{split}\label{milsteinS}
\end{align}
for the steps $n=0,\dots,N-1$, with $u_{n}\sim \mathcal{U}(0,1)$ i.i.d., $\widetilde{S}_0=S_0$, the cumulative standard normal distribution function $\Phi$, the survival probabilities
\begin{align}
\widetilde{p}_n&=\widetilde{p}^{+}_n- \widetilde{p}^{-}_n,\label{pcombined}
\end{align}
and with
\begin{align}
\widetilde{p}^{+,-}_n&=\Phi \left(\frac{-1\pm \sqrt{1+4\left(\frac{1}{2} \sigma'(\widetilde{S}_n,t_n) \right)\left(\frac{B-\widetilde{S}_n-\mu(\widetilde{S}_n,t_n)h +\frac{1}{2}\sigma(\widetilde{S}_n,t_n) \sigma'(\widetilde{S}_n,t_n)h}{\sigma(\widetilde{S}_n,t_n) }\right)}}{ \sigma'(\widetilde{S}_n,t_n) \sqrt{h}}\right).\label{milsteinp}
\end{align}
Furthermore, for each $n=0,\dots,N-1$ with $\sigma'(\widetilde{S}_n,t_n)=0$ consider the simplified scheme
\begin{align}
\widetilde{S}_{n+1}&=\widetilde{S}_{n}+\mu(\widetilde{S}_{n},t_n)h+ \sigma(\widetilde{S}_{n},t_n) \sqrt{h}\Phi^{-1}(\widetilde{p}_n u_{n})\label{modifiedscheme}
\end{align}
with
\begin{align}
\widetilde{p}_n&=\Phi \left( \frac{B-\widetilde{S}_n-\mu(\widetilde{S}_n,t_n)h}{\sigma(\widetilde{S}_n,t_n)\sqrt{h}} \right).\label{pforzerovol}
\end{align}
Then, the one-step survival Brownian bridge approximation defined by
\begin{align}
\begin{split}
\widetilde{P}\left(\widetilde{S}_N,\widehat{p}^*_0,\dots,\widehat{p}^*_{N-1},\widetilde{p}_0,\dots,\widetilde{p}_{N-1}\right):= q(\widetilde{S}_{N}) \prod\limits_{n=0}^{N-1} \left( 1- \widehat{p}^*_n\right) \prod\limits_{n=0}^{N-1}  \widetilde{p}_n \end{split},\label{IntBBOSS}
\end{align}
with
\begin{align}
\widehat{p}^*_n=\exp \left( \frac{-2(B-\widetilde{S}_n)(B-\widetilde{S}_{n+1})}{\sigma(\widetilde{S}_{n},t_n)^2 h}  \right).\label{pBB1}
\end{align}
satisfies
\begin{align}
\mathbb{E}[\widetilde{P}]&=\mathbb{E}[\widehat{P}]\label{ewgleich}
\end{align}
and
\begin{align}
\text{Var}[\widetilde{P}] &\le \text{Var}[\widehat{P}]\label{varreduction}.
\end{align}

\end{theorem}
\begin{proof}

First, we verify the equivalence of the expected values \eqref{BBint} and \eqref{IntBBOSS}. For simplicity, we will omit the cases $\sigma'(\widehat{S}_n,t_n)=0$ in the following derivation since they are applied straightforwardly.

Considering $N=2$, the expected value is given by
\begin{align*}
\mathbb{E}[\widehat{P}]=& \widetilde{p}_0\int\limits_{0}^{1}\left( 1- \widehat{p}^*_0\right) \left[\int\limits_{\widehat{S}_2(z)<B}\phi(z)  \left( 1- \widehat{p}_1\right)q\left(\widehat{S}_2(z,u)\right)\diff z\right]\diff u,
\end{align*}
with
\begin{align*}
\widehat{S}_2(z,u)=\widetilde{S}_{1}(u)&+\mu(\widetilde{S}_{1}(u),t_1)h+ \sigma(\widetilde{S}_{1}(u),t_1)z\\
&+ \frac{1}{2}\sigma(\widetilde{S}_{1}(u),t_1) \sigma'(\widetilde{S}_{1}(u),t_1)h\left( z^2-1\right).
\end{align*}
Again by splitting, substituting and using $\widetilde{p}_1=\widetilde{p}_1^{-}-\widetilde{p}_1^{+}$ we obtain
\begin{align*}
\mathbb{E}[\widehat{P}]&=\widetilde{p}_0\int\limits_{0}^{1}\left( 1- \widehat{p}^*_0\right) \int\limits_{\widehat{S}_2(z)<B}\phi(z)  \left( 1- \widehat{p}_1\right)q\left(\widehat{S}_2(z,u)\right)\diff z\diff u,\\
&=\widetilde{p}_0\int\limits_{0}^{1}\left( 1- \widehat{p}^*_0\right) \int\limits_{\Phi^{-1}(\widetilde{p}_1^{-})}^{\Phi^{-1}(\widetilde{p}_1^{+})}\phi(z)  \left( 1- \widehat{p}_1\right)q\left(\widehat{S}_2(z,u)\right)\diff z\diff u,\\
&=\widetilde{p}_0\int\limits_{0}^{1}\left( 1- \widehat{p}^*_0\right) \widetilde{p}_1\int\limits_{0}^{1}\left( 1- \widehat{p}^*_1\right)q\left(\widehat{S}_2(u_1,u_0)\right)\diff u_{1}\diff u_{0},
\end{align*}
with
\begin{align*}
\widetilde{S}_{2}(u_{1},u_{0})&=\widetilde{S}_1(u_{0})+\mu(\widetilde{S}_1(u_{0}),t_1)h+ \sigma(\widetilde{S}_1(u_{0}),t_1)\Phi^{-1}(\widehat{p}^{-}_1+\widetilde{p}_1u_{1})\\
&+\frac{1}{2}\sigma(\widetilde{S}_1(u_{0}),t_1) \sigma'(\widetilde{S}_1(u_{0}),t_1)h\left( (\Phi^{-1}(\widehat{p}^{-}_1+\widehat{p}_1 u_{1}))^2-1\right),\\
\widehat{p}_1^*&=\exp \left( \frac{-2(\widetilde{S}_1-B)(\widetilde{S}_{2}-B)}{\sigma(\widetilde{S}_1,t_1)^2 h}  \right)
\end{align*}
and $\hat{p}_0^*$ from \eqref{p0n1}.
For general $N$, we obtain
\begin{align}
\begin{split}
\mathbb{E}[\widehat{P}]&= \int\limits_{-\infty}^{\infty}\cdots \int\limits_{-\infty}^{\infty}\phi(z_{0})\cdots \phi(z_{N-1})  q(\widehat{S}_{N}(z_{N-1},\dots,z_{0}))  \prod\limits_{n=0}^{N-1} \left( 1- \widehat{p}_n\right)\diff z_{N-1}\cdots \diff z_{0},\\
&=\int\limits_{0}^{1}\cdots \int\limits_{0}^{1}q(\widetilde{S}_{N}(u_{N-1},\dots,u_{0}))\prod\limits_{n=0}^{N-1} \left( 1- \widehat{p}^*_n\right)\prod\limits_{n=0}^{N-1} \widetilde{p}_n  \diff u_{{N-1}}\cdots \diff u_{0}\\
&=\mathbb{E}[\widetilde{P}],\end{split}\label{ewgleich2}
\end{align}
with \eqref{milsteinp}, \eqref{pcombined}, \eqref{milsteinS} and \eqref{pBB1}.
To prove \eqref{varreduction}, we use analogue techniques leading to
\begin{align}
\begin{split}
\mathbb{E}[\widehat{P}^2]&= \int\limits_{-\infty}^{\infty}\cdots \int\limits_{-\infty}^{\infty}\phi(z_{0})\cdots \phi(z_{N-1}) q(\widehat{S}_{N}(z_{N-1},\dots,z_{0}))^2 \\
&\quad \quad \cdot \prod\limits_{n=0}^{N-1} \left( 1- \widehat{p}_n\right)^2 \diff z_{N-1}\cdots \diff z_{0},\\
&=\int\limits_{0}^{1}\cdots \int\limits_{0}^{1}q(\widetilde{S}_{N}(u_{N-1},\dots,u_{0}))^2\prod\limits_{n=0}^{N} \left( 1- \widehat{p}^*_n\right)^2\prod\limits_{n=0}^{N} \left( \widetilde{p}_n\right) \diff u_{{N-1}}\cdots \diff u_{0}\\
&\ge \int\limits_{0}^{1}\cdots \int\limits_{0}^{1}q(\widetilde{S}_{N}(u_{N-1},\dots,u_{0}))^2\prod\limits_{n=0}^{N} \left( 1- \widehat{p}^*_n\right)^2\prod\limits_{n=0}^{N} \left( \widetilde{p}_n\right)^2 \diff u_{{N-1}}\cdots \diff u_{0}\\
&=\mathbb{E}[\widetilde{P}^2],\end{split}\label{varungleich}.
\end{align}
which holds since we have $\widetilde{p}_n\in [0,1]$ and therefore, we have
\begin{align}
\prod\limits_{n=0}^{N}  \widetilde{p}_n^2\le \prod\limits_{n=0}^{N} \widetilde{p}_n.\label{varianzmeaning}
\end{align}
All in all, by using the unbiasedness of the first part, we obtain
\begin{align*}
\mbox{Var}[\widetilde{P}]=\mathbb{E}[\widetilde{P}^2]-\mathbb{E}[\widetilde{P}]^2=\mathbb{E}[\widetilde{P}^2]-\mathbb{E}[\widehat{P}]^2 \le \mathbb{E}[\widehat{P}^2]-\mathbb{E}[\widehat{P}]^2 =\mbox{Var}[\widehat{P}],
\end{align*}
which completes the proof.
\end{proof}

\begin{remark}
The variance reduction is a consequence of \eqref{varianzmeaning}, which is most significant near the barrier.
\end{remark}

\begin{remark}
Suppose we use the Euler-Maruyama scheme for the discretization of the Brownian bridge approach. In that case, the one-step survival Brownian bridge approach can be defined through a modified scheme, as in \eqref{modifiedscheme}, which will analogously lead to the unbiasedness and variance reduction property. Similar ideas hold for \cref{test2}: By omitting line 4 to 10, we can modify the algorithm to only use the Euler-Maruyama scheme for the discretization.
\end{remark}

\begin{corollary}
Consider an SDE \eqref{sde} with drift and positive volatility so that the Brownian bridge approximation \eqref{brownianp}, using a Milstein or Euler scheme converges with $\alpha>0$ in the sense of \eqref{alpha}. Then, the one-step survival Brownian bridge approach, using the respective modified scheme, satisfies convergence with order $\alpha$.
\end{corollary}
By forcing the path to stay below the barrier, we smooth out the indicator function, leading to a Lipschitz-continuous first derivative. As explained above, the Lipschitz-continuous first derivative will be the key point to achieve stable second-order Greeks. 
\begin{remark}
A modified substitution leads to an extension to down-and-out barrier options, as, e.g., explained in \cite{gerstner2018monte}. The extension to knock-in options is not straightforward. However, with the in-out parity, the pathwise sensitivities can be calculated through the pathwise sensitivities of knock-out barrier options and plain vanilla options.
\end{remark}

\subsection{Partial derivatives, pathwise sensitivities, and finite difference:\\ first-order Greeks}

In this section we will study different ways of calculating first-order Greeks for barrier options with payoff \eqref{payoffcont}.
Therefore, consider the stochastic differential equation
\begin{align}
dS(t)=\mu(S,t,a)\diff t + \sigma(S,t,b) \diff W(t), \text{ \hspace{3mm} } 0<t\le T,\label{sde-parameter}
\end{align}
with initial value $S(0)=s_0\in \mathbb{R}_+$, time of maturity $T\in\mathbb{R}_+$, drift $\mu(S,t,a)$ and volatility $\sigma(S,t,b)$, with $a=(a_1,\dots,a_m)\in \mathbb{R}^m$ and $b=(b_{1},\dots,b_{s})\in \mathbb{R}^s$.

Let the stochastic process ${S}_{a,b,s_0,T}$ be the solution of the SDE \eqref{sde-parameter} defined by the parameters $a,b,s_0$ and $T$. Consider a family of barrier payoff functions $V(S,v)$, with $v=(v_1,\dots,v_r) \in \mathbb{R}^r$. We are be interested in the expected value of
\begin{align}
P:\Theta \mapsto V\left({S}_{a,b,s_0,T}(T),v\right),\label{py}
\end{align}
for
\begin{align}
\Theta:=(a,b,v,s_0,T)\in Y  \subset \mathbb{R}^{V}\times \mathbb{R}_+^2,\label{training-intervall}
\end{align}
with $V=m+s+r$. From section 2.6.2 of \cite{burgosdis} we know that the derivatives of the survival probabilities \eqref{BBint} of the Brownian bridge approach for an up-and-out barrier option $\widehat{P}$ with respect to $\Theta$, assuming $\widehat{P}(\Theta)$ sufficiently regular in $\Theta$, are given by
\begin{align}
\frac{\partial \widehat{P}}{\partial \Theta_i}&= \left( \mathbf{1}_{\widehat{S}_{N}>K} \frac{\partial \widehat{S}_{N}}{\partial \Theta_i} \prod\limits_{n=0}^{N-1} \left(1-\widehat{p}_n\right)\right.\notag\\
&\hspace{10mm}\left.+ \left( \widehat{S}_{N}-K\right)^+ \sum\limits_{n=0}^{N-1}\left[ \prod\limits_{k=0,k \neq t}^{N-1} \left( 1- \widehat{p}_k(\Theta,u)\right) \frac{\partial \widehat{p}_n}{\partial \Theta_i} \right]\right).\label{burgosPV}
\end{align}
with 
\begin{align*}
\frac{\partial \widehat{p}_n}{\partial \Theta_i}= \mathbf{1}_{\widehat{S}_{n},\widehat{S}_{n+1}<B}\widehat{p}_n\left[ \frac{-2(B-\widehat{S}_{n+1})}{\sigma(\widehat{S}_{n},t_n)^2 h}+\frac{-2(B-\widehat{S}_n)}{\sigma(\widehat{S}_{n},t_n)^2 h}+\frac{4(B-\widehat{S}_{n+1})(B-\widehat{S}_{n})}{\sigma(\widehat{S}_{n},t_n)^3 h} \right]
\end{align*}
and $\frac{\partial \widehat{S}_n}{\partial \Theta_i}$, as similarly defined later. As mentioned above, we will see that the one-step survival Brownian bridge approach does not need the indicator functions, as in $\frac{\partial \widehat{p}_n}{\partial \Theta_i}$, leading to the following considerations. 
Before presenting the following theorem, we wish to remark that we can smooth out the indicator function arising in \eqref{burgosPV} by forcing the path to stay between $B$ and $K$ at the final step, see, e.g., \cite{gerstner2018monte} for further information. However, to keep the presentation simple, we omit the smoothing in the following theorem. Furthermore, we will use a more general notation for \eqref{milsteinS} to \eqref{pBB1}. 

\begin{theorem}\label{thm:OSSpathwisesensitivities}
The partial derivatives of the one-step survival Brownian bridge payoff $\widetilde{P}$ with respect to a vector of inputs $\Theta$, if $\widehat{S}(\Theta)$ is sufficiently regular in $\Theta$, are given by
\begin{align}
\begin{split}\frac{\partial \widetilde{P}}{\partial \Theta_i}
=&\left( \mathbbm{1}_{\widetilde{S}_{N}>K}\frac{\partial \widetilde{S}_{N}}{\partial \Theta_i} \prod\limits_{j=0}^{N-1}\widetilde{p}_j\prod\limits_{n=0}^{N-1} \left(1-\widehat{p}^*_n\right)\right.\\
 & \hspace{0mm}\left.+q(\widetilde{S}_{N}) \sum\limits_{j=0}^{N-1}\left[\frac{\partial \widetilde{p}_j}{\partial \Theta_i} \prod\limits_{k \neq j}^{N-1} \widetilde{p}_k\right]\prod\limits_{n=0}^{N-1} \left(1-\widehat{p}^*_n\right)\right.\\
&\left. - q(\widetilde{S}_{N}) \sum\limits_{n=0}^{N-1}\left[ \prod\limits_{k=0,k \neq t}^{N-1} \left( 1- \widehat{p}^*_k\right) \frac{\partial \widehat{p}^*_n}{\partial \Theta_i} \right]\prod\limits_{j=0}^{N-1}\widetilde{p}_j \right).\end{split}\label{pathwiseq}
\end{align}
where $\widetilde{P}, \widetilde{p}_n, \widehat{p}_n^*$, and $\widetilde{S}$ depend on $(\Theta,u)$. The derivatives of $\widetilde{p}^{-}_n(\Theta,u)$, $\widetilde{p}_n(\Theta,u)$, $\widetilde{S}_{n}(\Theta,u)$, $\widehat{p}_n(\Theta,u)$, ${\mu}_n(\Theta,u)$, ${\sigma}_n(\Theta,u)$, and ${\sigma}'_n(\Theta,u)$ are computed recursively as follows: 
Let $f_2(*_5)_{\left| \substack{*_2}\right.}:=\widetilde{p}_n^{-}(\Theta,u)$, with
\begin{align*}
*_5&:=(\nu,\varsigma,\varsigma',s,\vartheta)\\
*_2&:= \left(\nu=\mu(\widetilde{S}_n(\Theta,u),t_n),\varsigma=\sigma(\widetilde{S}_n(\Theta,u),t_n),\varsigma'=\sigma'(\widetilde{S}_n(\Theta,u),t_n),\right.\\ &\hspace{40mm} \left. s=\widetilde{S}_n(\Theta,u), \vartheta=\Theta\right).
\end{align*}
I.e., we will use the symbolic vector $*_5$ for the calculation of partial derivatives of $\widetilde{p}_n^{-}(\Theta,u)$ and $*_2$ will be the vector of inputs at the computation process. Then, the partial derivatives are recursively given by
\begin{align*}
\frac{\partial \widetilde{p}_n^{-}}{\partial \Theta_i}(\Theta,u)&=\frac{\partial f_2}{\partial s}(*_5)_{\left| \substack{*_2}\right.}\frac{\partial \widetilde{S}_n}{\partial \Theta_i}(\Theta,u)
+\frac{\partial f_2}{\partial \nu}(*_5)_{\left| \substack{*_2}\right.}\frac{\partial \mu_n}{\partial \Theta_i}(\Theta,u)\\
&+\frac{\partial f_2}{\partial \varsigma}(*_5)_{\left| \substack{*_2}\right.}\frac{\partial \sigma_n}{\partial \Theta_i}(\Theta,u)
+\frac{\partial f_2}{\partial \varsigma'}(*_5)_{\left| \substack{*_2}\right.}\frac{\partial \sigma_n'}{\partial \Theta_i}(\Theta,u)
+\frac{\partial f_2}{\partial \vartheta_i}(*_5)_{\left| \substack{*_2}\right.}.
\end{align*}
The remaining derivatives are calculated and simulated through the following: Let $h(*_4)_{\left|\substack*_1\right.}:=\widehat{p}_n^*(\Theta,u)$, $
f_2(*_5)_{\left| \substack{*_2}\right.}:=\widetilde{p}_n^{-}(\Theta,u)$, $
f(*_5)_{\left| \substack{*_2}\right.}:=\widetilde{p}_n(\Theta,u)$, and $
g(*_6)_{\left| \substack{*_3}\right.}:=\widetilde{S}_{n+1}(\Theta,u)$,
with
\begin{align*}
*_1&:= \left({\varsigma=\sigma(\widetilde{S}_n(\Theta,u),t_n), s_1=\widetilde{S}_n(\Theta,u),s_2=\widetilde{S}_{n+1}(\Theta,u), \vartheta=\Theta}\right),\\
*_3&:=\left( \nu=\mu(\widetilde{S}_n(\Theta,u),t_n), \varsigma=\sigma(\widetilde{S}_n(\Theta,u),t_n), \varsigma'=\sigma'(\widetilde{S}_n(\Theta,u),t_n), \pi=\widetilde{p}_n(\Theta,u),\right.\\ &\hspace{30mm}\left. \pi_2=\widetilde{p}_n^{-}(\Theta,u), s=\widetilde{S}_n(\Theta,u), \vartheta=\Theta,\omega=u_{N-1}\right),\\
*_4&:=(\varsigma,s_1,s_2,\vartheta)\\
*_6&:=(\nu,\varsigma,\varsigma',\pi,\pi_2,s,\vartheta,\omega)
\end{align*}
and $u=(u_{N-1},\dots,u_{0})$. Then, the partial derivatives are recursively given by

\begin{align*}
\frac{\partial \widetilde{p}_n}{\partial \Theta_i}(\Theta,u)&=\frac{\partial f}{\partial s}(*_5)_{\left| \substack{*_2}\right.}\frac{\partial \widetilde{S}_n}{\partial \Theta_i}(\Theta,u)
+\frac{\partial f}{\partial \nu}(*_5)_{\left| \substack{*_2}\right.}\frac{\partial \mu_n}{\partial \Theta_i}(\Theta,u)
+\frac{\partial f}{\partial \varsigma}(*_5)_{\left| \substack{*_2}\right.}\frac{\partial \sigma_n}{\partial \Theta_i}(\Theta,u)\\
&+\frac{\partial f}{\partial \varsigma'}(*_5)_{\left| \substack{*_2}\right.}\frac{\partial \sigma_n'}{\partial \Theta_i}(\Theta,u)
+\frac{\partial f}{\partial \vartheta_j}(*_5)_{\left| \substack{*_2}\right.}\\
\frac{\partial \widehat{p}_n^*}{\partial \Theta_i}(\Theta,u)&=\frac{\partial h}{\partial s_1}(*_4)_{\left| \substack{*_1}\right.}\frac{\partial \widetilde{S}_n}{\partial \Theta_i}(\Theta,u)
+\frac{\partial h}{\partial s_2}(*_4)_{\left| \substack{*_1}\right.}\frac{\partial \widetilde{S}_{n+1}}{\partial \Theta_i}(\Theta,u)\\
&+\frac{\partial h}{\partial \varsigma}(*_4)_{\left| \substack{*_1}\right.}\frac{\partial \sigma_n}{\partial \Theta_i}(\Theta,u)
+\frac{\partial h}{\partial \vartheta_i}(*_4)_{\left| \substack{*_1}\right.},\\
\frac{\partial \widetilde{S}_{n+1}}{\partial \Theta_i}(\Theta,u)&=\frac{\partial g}{\partial \nu}(*_6)_{\left| \substack{*_3}\right.}\frac{\partial \mu_n}{\partial \Theta_i}(\Theta,u)
+\frac{\partial g}{\partial \varsigma}(*_6)_{\left| \substack{*_3}\right.}\frac{\partial \sigma_n}{\partial \Theta_i}(\Theta,u)
+\frac{\partial g}{\partial \varsigma'}(*_6)_{\left| \substack{*_3}\right.}\frac{\partial \sigma_n'}{\partial \Theta_i}(\Theta,u)\\
&+\frac{\partial g}{\partial \pi}(*_6)_{\left| \substack{*_3}\right.}\frac{\partial \widetilde{p}_n}{\partial \Theta_i}(\Theta,u)
+\frac{\partial g}{\partial \pi_2}(*_6)_{\left| \substack{*_3}\right.}\frac{\partial \widetilde{p}_n^{-}}{\partial \Theta_i}(\Theta,u)
+\frac{\partial g}{\partial s}(*_6)_{\left| \substack{*_3}\right.}\frac{\partial \widetilde{S}_n}{\partial \Theta_i}(\Theta,u)\\
&+\frac{\partial g}{\partial \vartheta_i}(*_6)_{\left| \substack{*_3}\right.}.
\end{align*}
The derivatives of the local drift and volatility are given by
\begin{align*}
\frac{\partial {\mu}_n}{\partial \Theta_i}(\Theta,u)&=\frac{\partial k}{\partial s}(s,\vartheta)_{\left| \substack{*_7}\right.}\frac{\partial \widetilde{S}_n}{\partial \Theta_i}(\Theta,u)+ \frac{\partial k}{\partial \vartheta_i}(s,\vartheta)_{\left| \substack{*_7}\right.},\\
\frac{\partial {\sigma}_n}{\partial \Theta_i}(\Theta,u)&=\frac{\partial l}{\partial s}(s,\vartheta)_{\left| \substack{*_7}\right.}\frac{\partial \widetilde{S}_n}{\partial \Theta_i}(\Theta,u)+ \frac{\partial l}{\partial \vartheta_i}(s,\vartheta)_{\left| \substack{*_7}\right.},\\
\frac{\partial {\sigma}_n'}{\partial \Theta_i}(\Theta,u)&=\frac{\partial m}{\partial s}(s,\vartheta)_{\left| \substack{*_7}\right.}\frac{\partial \widetilde{S}_n}{\partial \Theta_i}(\Theta,u)+ \frac{\partial m}{\partial \vartheta_i}(s,\vartheta)_{\left| \substack{*_7}\right.},
\end{align*}
with $*_7:=\left(s=\widetilde{S}_n(\Theta,u), \vartheta=\Theta\right)$.

\end{theorem}
\begin{proof}
\eqref{pathwiseq} is the derivative of the integrand of \eqref{IntBBOSS}.
For \eqref{milsteinp}, \eqref{milsteinS}, \eqref{pcombined} and \eqref{pBB1} we have
\begin{align*}
f_2(*_5)&=\Phi \left(\frac{-1- \sqrt{1+4\left(\frac{1}{2} \varsigma' \sqrt{\vartheta_1}\right)\left(\frac{\vartheta_2-s-\nu\vartheta_1 +\frac{1}{2}\varsigma \varsigma'\vartheta_1}{\varsigma \sqrt{\vartheta_1}}\right)}}{ \varsigma' \sqrt{\vartheta_1}}\right), \\
f(*_5)&=\Phi \left(\frac{-1+ \sqrt{1+4\left(\frac{1}{2} \varsigma' \sqrt{\vartheta_1}\right)\left(\frac{\vartheta_2-s-\nu\vartheta_1 +\frac{1}{2}\varsigma \varsigma'\vartheta_1}{\varsigma \sqrt{\vartheta_1}}\right)}}{ \varsigma' \sqrt{\vartheta_1}}\right), \\
&\quad -\Phi \left(\frac{-1- \sqrt{1+4\left(\frac{1}{2} \varsigma' \sqrt{\vartheta_1}\right)\left(\frac{\vartheta_2-s-\nu\vartheta_1 +\frac{1}{2}\varsigma \varsigma'\vartheta_1}{\varsigma \sqrt{\vartheta_1}}\right)}}{ \varsigma' \sqrt{\vartheta_1}}\right),\\
g\left(*_6\right)&=s+\nu \vartheta_1+ \varsigma \sqrt{\vartheta_1}\Phi^{-1}\left(\pi_2+\pi  \omega\right),\\
&\quad + \frac{1}{2}\varsigma \varsigma'\vartheta_1\left( (\Phi^{-1}\left(\pi_2+\pi  \omega\right))^2-1\right),\\
h(*_4)&=\exp \left( \frac{-2(\vartheta_2-s_1)(\vartheta_2-s_2)}{\varsigma^2\vartheta_1}\right).
\end{align*}
The recursive formulas follow through differentiation with the product rule.
\end{proof}

\begin{lemma}
Consider an SDE \eqref{sde} with drift and positive volatility such that the partial derivatives of the Brownian bridge approximation \eqref{brownianp}, using the Milstein- or Euler-Maruyama scheme, converges with $\alpha>0$ in the sense of \eqref{alpha}. Then, the partial derivatives of the one-step survival Brownian bridge approach satisfy convergence with order $\alpha$.
\end{lemma}
\begin{proof}
Since having compact domains in \eqref{ewgleich2} and a Lipschitz-continuous integrand, the interchange of differentiation and expectation is justified for the one-step survival Brownian bridge approach. For the Brownian bridge approach, we can substitute the integral with $z=\Phi^{-1}(u)$ in each discretization step to obtain a Lipschitz-continuous integrand and compact domains. Thus, the one-step survival Brownian bridge derivatives are unbiased with respect to the Brownian bridge derivatives.
\end{proof}

For a convergence result on the Brownian bridge derivatives, assuming certain assumptions, we refer to section 7 of \cite{burgosdis}.
For both \eqref{burgosPV} and \eqref{pathwiseq}, one could formulate corollaries of an unbiased pathwise sensitivity Monte Carlo estimator. We remark that one can automate most of the calculations, e.g., as in \cite{gerstner2018monte} by MATLAB.

For a more accessible alternative, the first-order Greeks can be calculated with finite differences under a specific stability condition. Nevertheless, finite differences add an error source and should be used carefully.
\begin{definition}\label{stable}
We say that a Monte Carlo estimator $\overline{P_{M}}$ with Monte Carlo payoff $\mathcal{Q}$ allows for stable differentiation by finite differences if there exists $C>0$ such that 
\begin{align*}
\text{Var}\left( D_h \overline{P_{M}}\right) \le \frac{1}{M}C
\end{align*}
with a positive constant $C$.
\end{definition}
In theorem 2.2 of \cite{alm}, the authors show that if both $\widehat{P}$ and its Monte Carlo payoff depend Lipschitz-continuously on $\Theta$, the estimator allows for stable differentiation with respect to $\Theta$. This is the case for the Monte Carlo estimators $\overline{P_M}$ for both the Brownian bridge approximation $\widehat{P}$ and the one-step survival Brownian bridge approximation $\widetilde{P}$.

\subsection{Pathwise sensitivities, and finite difference: second-order Greeks}

In this section, we study three different ways to obtain second-order Greeks for barrier options. As already mentioned, \eqref{burgosPV} is not Lipschitz-continuous, and hence it is firstly not differentiable and secondly does not apply for theorem 2.2 of \cite{alm}. The mentioned theorem would imply stable second-order Greeks through the first-order Greeks' finite differences if the first-order Greeks were Lipschitz-continuous. However, we have a, at least twice, continuously differentiable payoff, for the one-step survival Brownian bridge approximation. One could calculate the Greeks through pathwise sensitivities, which can be done by a straightforward extension of \cref{thm:OSSpathwisesensitivities}. 

Alternatively, the theorem 2.2 of \cite{alm} can be applied, i.e., one could use finite differences of the first-order Greeks (gained through pathwise sensitivities or finite differences). We present a third alternative in the following theorem, which uses second-order finite differences and is noted quite generally.
\begin{theorem}\label{thm:fin}
If all, $PV_{t_0}(\Theta),PV_{t_0}'(\Theta),\mathcal{Q}(\Theta,u)$, and $\mathcal{Q}'(\Theta,u)$, depend Lipschitz-continuously on $\Theta$, resp., $(\Theta,u)$, then the Monte Carlo estimator  allows for stable second-order differentiation by means of \cref{stable}.
\end{theorem}
\begin{proof} 
For an estimator in the form \eqref{mc2} and $U \sim U(0,1)^N  $, we have that
\begin{align}
\text{Var}\left(D_h^{(2)}{\overline{P_{M}}(\Theta)}\right) 
&=\frac{1}{M}\text{Var}\left(D_h^{(2)}\mathcal{Q}(\Theta,U)\right)\notag\\
&\le\frac{1}{M}\int_{(0,1)^N} \left( D_h^{(2)} \mathcal{Q}(\Theta,u) -D_h^{(2)} PV_{t_0}(\Theta) \right)^2 \diff u\notag\\
&\le \frac{1}{M}\int_{(0,1)^N} \left(\left|D_h^{(2)} \mathcal{Q}(\Theta,u)\right| + \left|D_h^{(2)} PV_{t_0}(\Theta) \right|\right)^2 \diff u.\label{integrand}
\end{align}
For the left-hand side of the integrand, we obtain
\begin{align*}
D_h^{(2)} \mathcal{Q}(\Theta,u) &= \frac{\mathcal{Q}(\Theta+h,u)-2\mathcal{Q}(\Theta,u)+\mathcal{Q}(\Theta-h,u)}{h^2}\\
&=\frac{\frac{1}{h}\int_{0}^{1}\mathcal{Q}'(\Theta+ht,u)h\diff t-\frac{1}{h}\int_{0}^{1}\mathcal{Q}'(\Theta-ht,u)h\diff t}{h}\\
&\le \frac{\int_{0}^{1}\left| \mathcal{Q}'(\Theta+ht,u)-\mathcal{Q}'(\Theta-ht,u)\right|\diff t}{h} \le \frac{L\left| 2h\right|}{h}\le C.
\end{align*}
Analogue relations hold for the right-hand side of the integrand in \eqref{integrand}. Together, we obtain
\begin{align*}
\frac{1}{M}\int_{(0,1)^N} \left( \left| D_h^{(2)} \mathcal{Q}(\Theta,u)\right| + \left|D_h^{(2)} PV_{t_0}(\Theta) \right|\right)^2 \diff u \le \frac{1}{M} C,
\end{align*}
which completes the proof.
\end{proof}

\section{Multilevel one-step survival Monte Carlo method}
This section will present a multilevel algorithm for the one-step survival Brownian bridge Monte Carlo estimator. We will briefly explain the issues arising from the multilevel approach and how to overcome them. For a more in-depth introduction to the multilevel approach, we refer to \cite{giles2008multilevel}.
For the complexity theorem of Giles \cite{giles2008improved}, one wishes most of all to have \textit{level estimators} with variance $V[\hat{Y}_l]\le cM^{-1}_lh_l^\beta$ of order $\beta>1$, for a positive constant $c$, $M_l$ simulations on level $l=0,\dots,L$.
From Giles \cite{giles2008improved}, we know that standard implementation of the Brownian bridge approach and the use of the Milstein scheme numerically leads to $\beta \approx 0.5$. Furthermore, Giles introduces a path modification leading to numerically $\beta\approx 1.5$. However, for the one-step survival Brownian bridge approximation, we have to overcome the issue of not using this path modification technique since the used midterm interpolation would lead to biased one-step survival probabilities. Nevertheless, we found a way to modify the one-step survival Brownian bridge coarse path simulation so that it numerically achieved $\beta\approx 1.5$. The algorithm of the procedure can be found in \cref{test3}.

\begin{algorithm}\label{algorithm3}
\caption{Multilevel one-step survival Brownian bridge Monte Carlo coarse path computation}\label{test3}
\begin{algorithmic}[1]
\STATE Initialize random seed
\FOR {$n=0:N-1$}

\STATE $\widetilde{p}_{n}=\Phi\left(\frac{\frac{1}{2}\nu_n\sqrt{\frac{\sigma_n^2 h - 2 \nu_n  (\widetilde{S}_n+\mu h - \frac{1}{2}\nu_n - B )}{(\frac{1}{2}\nu_n)^2}}-\sigma_n \sqrt{h}}{\nu_n}\right)$
\STATE Sample $u_1 \sim U(0,1)$
\STATE $\widetilde{S}_{n+\frac{1}{2}}=\widetilde{S}_n+\mu  h + \sigma_n \sqrt{h} \Phi^{-1}(p_{n} u_1)+\nu_n \left(\left(  \Phi^{-1}(p_{n} u_1)\right)^2-1\right)$
\STATE $\widetilde{p}_{n+\frac{1}{2}}=\Phi\left(\frac{-\sigma_n \sqrt{h}-\nu_n \Phi^{-1}(u_1) + \sqrt{\left(\sigma_n  \sqrt{h}+\nu_n\right)^2-\nu_n \left(\widetilde{S}_{n+\frac{1}{2}}+\mu_n  h -B +\frac{1}{2}\nu_n\right)}}{\nu_n}\right)$
\STATE Sample $u_2 \sim U(0,1)$
\STATE $\widetilde{S}_{n+1}=\widetilde{S}_{n+\frac{1}{2}}+ \mu_n   h + \sigma_n  \sqrt{ h}\Phi^{-1}( u_2)$
\STATE $ \hspace{8mm}+\frac{1}{2} \nu_n \left( 2\Phi^{-1}(u_1)\Phi^{-1}(\widetilde{p}_{n+\frac{1}{2}} u_2)+(\Phi^{-1}(\widetilde{p}_{n+\frac{1}{2}} u_2))^2-1\right)$

\STATE $\widetilde{p}_{n}=\widetilde{p}_{n}\cdot \widetilde{p}_{n+\frac{1}{2}}$

\STATE $\widehat{p}^*_n=\left(1-\exp \left( \frac{-2(B-\widetilde{S}_{n+\frac{1}{2}})(B-\widetilde{S}_{n+1})}{\sigma(\widetilde{S}_{n},t_n)^2 h}  \right)\right) \left(1-\exp \left( \frac{-2(B-\widetilde{S}_n)(B-\widetilde{S}_{n+\frac{1}{2}})}{\sigma(\widetilde{S}_{n},t_n)^2 h}  \right)\right)$  
\ENDFOR
\STATE \textbf{return} $\widetilde{S}, \widetilde{p}, \widehat{p}^*$
\end{algorithmic}
\end{algorithm}

In the following we will explain \cref{test3}: First of all, the algorithm only  considers the (more complex) case $\sigma_n'(\widetilde{S}_n,t_n)\neq 0$. Furthermore, for an easier readability, we consider $\widetilde{p}^{-}_n=0$. The extension to general $\sigma_n'(\widetilde{S}_n,t_n)$ and $\widetilde{p}^{-}_n$ is straightforward as described in \cref{test2}. 
We use a simplified notation, i.e. we use $\sigma_n:=\sigma(\widetilde{S}_n,t_n)$, $\mu_n:=\mu(\widetilde{S}_n,t_n)$, $\sigma_n':=\sigma_n'(\widetilde{S}_n,t_n)$ and $\nu_n=\sigma_n\sigma_n'h$. 
Instead of one coarse step, using the one-step survival Brownian bridge discretization, with step-width $2h$, the algorithm computes two steps with step-width $h$. However, it slightly differs from the fine path simulation: We reuse the random sample $u_1$ used for the first step (line 5) to simulate the second step (line 9). Furthermore, we do not use $\sigma_{n+\frac{1}{2}}$ or  $\mu_{n+\frac{1}{2}}$ at any time. That means, even though two discretization steps are used, the discretization error of the coarse step-width $2h$ prevails. Finally, in line 11, the Brownian bridge probability is applied for both $S_{n+\frac{1}{2}}$ and $S_{n+1}$.

For a better understanding of this modification, we will give a short derivation of the coarse modification. Starting with the unmodified Milstein scheme on the coarse path using, e.g., step-width $2h$, we have:
\begin{align*}
\widehat{S}_{n+1}&=\widehat{S}_n+\mu_n  2 h + \sigma_n  \sqrt{2 h} Z_{n+1}+ \frac{1}{2}\nu_n  \left(\left( \sqrt{2 } Z_{n+1}\right)^2-2 \right).
\end{align*}
Now, using the relation $Z_{n+1}=\frac{Z_n+Z_{n+\frac{1}{2}}}{\sqrt{2}}$, we obtain
\begin{align*}
\widehat{S}_{n+1}&=
\widehat{S}_n+\mu_n  h+\sigma_n \sqrt{ h} Z_n+\frac{1}{2}\nu_n \left( \left(Z_n\right)^2-1\right)\notag\\
& \qquad \,\,\,+ \mu_n  h +\sigma_n \sqrt{ h} Z_{n+\frac{1}{2}}+\frac{1}{2} \nu_n \left( 2Z_nZ_{n+\frac{1}{2}}+Z_{n+\frac{1}{2}}^2-1\right).\label{help1}
\end{align*}
Now, we denote the first part of the above expression to be $S_{n+\frac{1}{2}}$, leading to
\begin{align*}
\widehat{S}_{n+1}&
=S_{n+\frac{1}{2}}+ \mu_n  h +\sigma_n \sqrt{ h} Z_{n+\frac{1}{2}}+\frac{1}{2} \nu_n \left( 2Z_nZ_{n+\frac{1}{2}}+(Z_{n+\frac{1}{2}})^2-1\right).
\end{align*}
Here, the differences between the two fine discretization steps become quite clear: We reuse the first step's random variable, and the drift and volatility are evaluated at $n$. Applying the Brownian bridge crossing probability to $S_{n+\frac{1}{2}}$ with $Z_n$ and to $S_{n+1}$ with $Z_{n+\frac{1}{2}}$ while assuming that $Z_n$ is a constant, leads to \cref{test3}.

\section{Numerical Results}

This section will provide some numerical results for the one-step survival Brownian bridge estimator and its derivatives. Therefore, we consider a simple, continuously observed, up-and-out barrier option. We will use parameters, as presented in \cref{parameter}, whereby the example is fictitious.

\begin{table}
\centering
  \begin{tabular}{lr}
  \hline
 Parameter & Value\\
 \hline
 $t_0$ & 0 \\
 $T$ & 1  \\
 $S_0$ & $1$  \\
 $B$ & $1.1$ \\
 $r$ & 5 \%\\
 $b$ & 0 \%\\
 $\sigma$ & $20$ \%\\
 $K$ & $1$ \\
 \hline
 \end{tabular}
 \caption{Parameters used for the simulation of the up-and-out barrier option.}
  \label{parameter}
\end{table}

In the first column of \cref{fig:vergleich0}, we see the estimated mean squared error of the options' present value with respect to the Monte Carlo samples. In the second column, we see the estimated absolute error with respect to the calculation time. The Brownian bridge estimator results and the one-step survival Brownian bridge estimator are plotted in a blue line and a red line. We observe the proven variance reduction, which depends on \eqref{varianzmeaning}, as mentioned above. However, the computation time for a specific error $\epsilon$ is similar since, for the one-step survival Brownian bridge approach, more terms have to be evaluated. Nevertheless, one could observe better results for initial values nearer to the barrier.

\begin{figure}
\centering
\includegraphics[scale=.5]{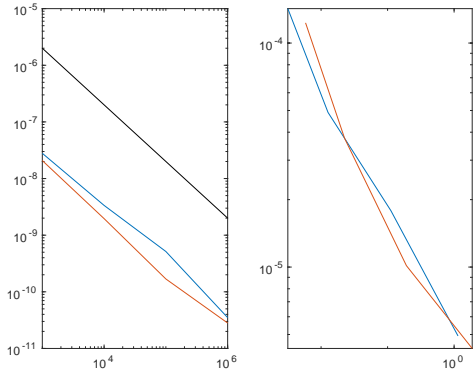}
\vspace{-7mm}
\caption{The figures show the mean squared error of the present value calculation of the one-step survival Brownian bridge estimator (red line) and the Brownian bridge estimator (blue line), depending on the amount of Monte Carlo simulations on the left side and calculation time (CPU) on the right.}
\label{fig:vergleich0}
\end{figure}

Next, we want to take a more in-depth look at comparing the two estimators for the sensitivity calculation. Therefore, we analogously compare the mean squared error and the calculation time in \cref{fig:vergleich1}, but on this occasion, the options' Delta is calculated through the pathwise sensitivity approach. The results are again plotted in a blue line for the Brownian bridge estimator and in a red line for the one-step survival Brownian bridge estimator.

Here, we see the strength and advantage of the discussed properties of the first derivative of the one-step survival Brownian bridge estimator, leading to a significant variance reduction and time savings.

\begin{figure}
\centering
\includegraphics[scale=.45]{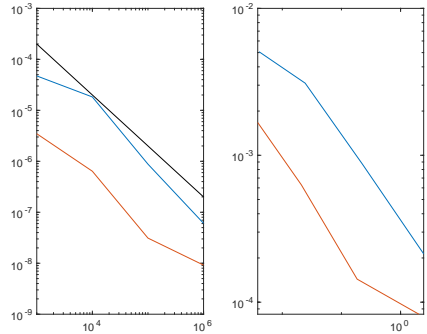}
\vspace{-7mm}
\caption{The figures show the absolute error of the Delta calculation of the one-step survival Brownian bridge estimator (red line) and the Brownian bridge estimator (blue line), depending on the amount of Monte Carlo simulation on the left side and calculation time (CPU) on the right.}
\label{fig:vergleich1}
\end{figure}

Now, we study the stability of the second-order Greeks. \cref{fig:vergleich} shows the second derivative of the barrier options present value with respect to the underlying price (the Gamma) calculated by applying second-order finite differences as in \cref{thm:fin}, to both the Brownian bridge estimator and the one-step survival Brownian bridge estimator plotted in a blue line and a red line, respectively. The plot demonstrates the Brownian bridge estimator's instability with respect to second-order numerical differentiation and the stability of the Brownian bridge one-step survival estimator.

\begin{figure}
\centering
\includegraphics[scale=.45]{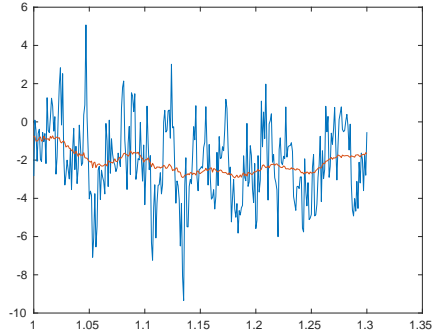}
\vspace{-7mm}
\caption{This figure shows the second-order Greek (Gamma) of the one-step survival Brownian bridge estimator (red line) and the Brownian bridge estimator (blue line), depending on the initial values of $S_0$. Gamma was calculated through second-order finite differences with step width $h=10^{-3}$ of $S_0$ and $M=10^5$ Monte Carlo simulations.}
\label{fig:vergleich}
\end{figure}

Lastly, we study the properties of the introduced multilevel modification. \cref{fig:beta} shows the behavior of both $\widetilde{P}_l$ and $\widetilde{P}_l-\widetilde{P}_{l-1}$, with the logarithmic base $2$ as quantity versus the grid level. The slope of the line for $\widetilde{P}_l-\widetilde{P}_{l-1}$ is approximately $1.5$, indicating the wished beta.

\begin{figure}
\centering
\includegraphics[scale=.45]{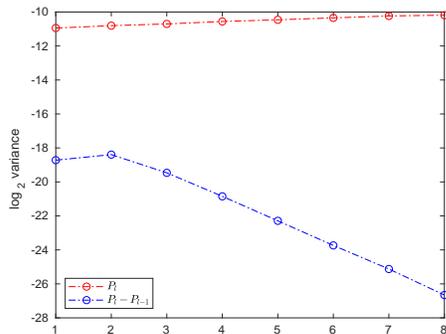}
\vspace{-7mm}
\caption{The plot shows the behaviour of the variance of both $\hat{P}_l$ (red line) and $\hat{P}_l-\hat{P}_{l-1}$ (blue line).}
\label{fig:beta}
\end{figure}

\section{Conclusion}

We faced the problem that the Brownian bridge approximation leads to a non-Lipschitz-continuous first derivative. As seen in the numerical results, this leads to the Brownian bridge estimator's instability for the second-order Greek computation.
We adapted the one-step survival Monte Carlo method to the Milstein scheme and the Brownian bridge approach to overcome this issue. The adaption resulted in a new one-step survival Brownian bridge approximation with a modified Milstein scheme and the slightly modified/smoothed crossing probabilities of the Brownian bridge interpolation. We showed that this new approach is unbiased with respect to the Brownian bridge approach and leads to variance reduction.
 
Furthermore, we presented the one-step survival Brownian bridge partial derivatives and showed unbiasedness. In a numerical experiment, we saw huge variance reductions for first-order Greeks. We theoretically showed that finite differences could stably differentiate it. Furthermore, we presented a numerical example showing this property.

We only presented first-order pathwise sensitivity results to simplify the presentation, but one could straightforwardly extend these results to the new approximation's second-order pathwise sensitivities.

Keeping the computational complexity in mind, we also provided a multilevel Monte Carlo algorithm and demonstrated its computational efficiency.

Even if it is restricted to up-and-out barrier call options, the conversion to put or down-and-out options is straightforward. Furthermore, we expect that the algorithms could be extended to the multivariate case, similar to the ideas of \cite{giles2013numerical} and \cite{alm}.

In the appendix, we showed weak convergence of almost one and a variance bound for the Brownian bridge approach. The proven (weak) convergence requires certain assumptions on the stochastic differential equations using the Milstein scheme. 

Finally, it should be mentioned that the new approach can be combined with other variance reduction methods as well, such as antithetic sampling or control variates \cite{staum}.

\bibliographystyle{siamplain}
\bibliography{references}

\newpage

\appendix
\section{Convergence and variance bound for the Brownian bridge and the one-step survival Brownian bridge approximation}

This section provides some theoretical background and proves a convergence property for the Brownian bridge approximation.
It is well known, see, e.g., theorem 4.5.3 of \cite{kloeden2012numerical}, that, under certain conditions, an SDE of the form \eqref{sde} has a pathwise unique, strong solution $S(t)$ on $[0,T]$.
Allowing an application of the Milstein scheme, we assume $\mu \in \mathcal{C}^{1,1}(\mathbb{R}\times \mathbb{R}^+)$ and $\sigma \in \mathcal{C}^{2,1}(\mathbb{R}\times \mathbb{R}^+)$. Furthermore, consider the following assumptions for all $x,y,t,s$ and with $L_0\equiv \partial /\partial t + \mu \partial / \partial x$ and $L_1 \equiv \partial / \partial x$:
\begin{itemize}
\item A1\label{a1} (uniform Lipschitz-condition): There exists a constant $K_1>0$ such that
\begin{align*}
| \mu(x,t)-\mu(y,t)| + | \sigma(x,t)-\sigma(y,t)| + |L_1\sigma(x,t)-L_1\sigma(y,t)| \le K_1 |x-y|.
\end{align*}
\item A2 (linear growth bound): There exists a constant $K_2>0$ such that
\begin{align*}
|\mu(x,t)|+|L_0 \mu(x,t)|+|L_1\mu(x,t)|+|\sigma(x,t)|+|L_0\sigma(x,t)|\\
+|L_1\sigma(x,t)| + | L_0L_1 \sigma (x,t)|+|L_1L_1\sigma(x,t)| \le K_2 (1+|x|).
\end{align*}
\item A3 (additional Lipschitz-condition): There exists $K_3>0$ such that
\begin{align*}
|\sigma(x,t)-\sigma(x,s)| \le K_3 (1+|x|) \sqrt{|t-s|}.
\end{align*}
\item A4: $\sigma_{\min}\equiv \inf_{[0,T]}|\sigma(B,t)|>0$.
\item A5: $\inf_{[0,T]}S(t)$ has a bounded density in the neighbourhood of $B$.
\end{itemize}

\begin{theorem}\label{thmconvergence}
Provided that assumptions A1 to A5 are satisfied and the Milstein scheme is used, as described in \eqref{milstein}, the Brownian bridge approximation \eqref{BBint} satisfies
\begin{align*}
\mathbb{E}[P-\widehat{P}]  &< c h^{1-\delta},\\
\mbox{Var}[\widehat{P}]&<\infty,
\end{align*}
for any positive $c$ and $\delta$.
\end{theorem}
\begin{proof}
See \cref{proofsec}.
\end{proof}

\begin{corollary}
The Brownian bridge Monte Carlo estimator for the present value of an up-and-out barrier option, using the Milstein scheme and given by \cref{mc1}, satisfies $\mathbb{E}[P-\overline{P_{M}}]  \leq c h^{1-\delta}$, for any positive $c$ and $\delta$. Furthermore, it satisfies $\mbox{Var}[\overline{P_{M}}]\le C/M$, for a positive constant $C$.
\end{corollary}

\subsection{proof of \cref{thmconvergence}}\label{proofsec}
In all the cases, the paths are discretized by using the Milstein scheme. As explained above, the Brownian bridge approximation \eqref{BBint} samples if the maximum exceeds the barrier between two discretization steps. From, e.g., section 6.4. of \cite{glasserman}, we know that the crossing probability \eqref{brownianp} would be accurate if facing an SDE with constant drift and volatility. However, we have to study the discretisation error's impact on these probabilities for general SDE's. For the proof, we first introduce some known results and present some lemmata.

\begin{lemma}\label{thm:originalS}
Provided assumptions A1 to A3 are satisfied, then for all positive integers $m$ there exists a constant $C_m$ such that
\begin{align*}
\mathbb{E}\left[ \sup\limits_{0\leq t\leq T} |S(t)|^m \right] < C_m.
\end{align*}
\end{lemma}
\begin{proof}
See theorem 10.6.3. and Corollary 10.6.4. in \cite{kloeden2012numerical}.
\end{proof}

\begin{definition}
The Kloeden \& Platen continuous-time interpolant for $t_n\le t \le t_{n+1}$ is defined by
\begin{align*}
\widehat{S}_{KP}(t )=
\widehat{S}_{n}+\mu(\widehat{S}_{n},{t_n})(t -t_n)+ \sigma(\widehat{S}_{n},{t_n}) (W(t)-W_{t_n})\\ + \frac{1}{2}\sigma(\widehat{S}_{n},{t_n}) \sigma'(\widehat{S}_{n},{t_n})\left( (W(t)-W_{t_n})^2-(t-t_n)\right).
\end{align*}
\end{definition}

\begin{lemma}\label{giles3.2}
Provided assumptions A1 to A3 are satisfied, then for all positive integers $m$ there exists a positive constant $C_m$ such that
\begin{align*}
\mathbb{E}\left[ \sup\limits_{0\le t \le T} | S(t)-\widehat{S}_{KP}(t) |^m\right] &<C_mh^m,\\
 \mathbb{E}\left[ \sup\limits_{0\le t \le T} | \widehat{S}_{KP}(t) |^m \right] &<C_m,
\end{align*}
with $h={t_{n+1}}-{t_n}$.
\end{lemma}
\begin{proof}
See theorem 10.6.3. and Corollary 10.6.4. in \cite{kloeden2012numerical}.
\end{proof}

\begin{lemma}\label{thm:boundeurop}
Provided assumptions A1 to A3 are satisfied, then the approximation of a European call option, given by
\begin{align*}
\widehat{P}^{\mbox{europ.}}:= q(\widehat{S}_{N}),
\end{align*}
satisfies that for all integers m there exists a positive constant $C_m$ such that
\begin{align*}
\limsup \limits_{h \downarrow 0 } \mathbb{E}\left[ |  \widehat{P}^{\mbox{europ.}} |^m\right] < C_m.
\end{align*}

\end{lemma}
\begin{proof}
We have that $\widehat{P}^{\mbox{europ.}}$ is Lipschitz-continuous with $L>0$. By using \cref{giles3.2}, we obtain
\begin{align*}
\limsup \limits_{h \downarrow 0 } \mathbb{E}\left[ |  \widehat{P}^{\mbox{europ.}} |^m\right] &\le \limsup \limits_{h \downarrow 0 } L^m  \mathbb{E}\left[ |\widehat{S}_{N} |^m\right]\\
& \le  L^m  \mathbb{E}\left[ \limsup \limits_{h \downarrow 0 }|\widehat{S}_{N} |^m\right]\\
&\le L^m \mathbb{E}\left[ \lim \limits_{h \downarrow 0 }\left(\sup\limits_{0\le t \le T} | \widehat{S}_{KP}(T) |^m\right) \right]< L^m C_m
\end{align*}
Since on $t_n$ and $t_{n+1}$ the Milstein scheme's discretisation steps equal the Kloeden \& Platen interpolant, the inequality holds.
\end{proof}


\begin{lemma}\label{thm:oss}
Provided assumptions A1 to A3 are satisfied, then the Brownian bridge approximation \eqref{BBint} satisfies that for all integers m there exists a positive constant $C_m$ such that
\begin{align}
\limsup \limits_{h \downarrow 0 } \mathbb{E}\left[ |\widehat{P}|^m \right] < C_m.
\end{align}
\end{lemma}
\begin{proof}
Since $\widehat{p}_n \in [0,1]$, we have $\limsup_{h \downarrow 0 }\prod_{n=0}^{N-1} ( 1- \widehat{p}_n) \le 1$. Together with \cref{thm:boundeurop}, we obtain
\begin{align*}
\limsup \limits_{h \downarrow 0 } \mathbb{E}\left[ |\widehat{P}|^m \right]&=\limsup \limits_{h \downarrow 0 } \mathbb{E}\left[\left|\left( \prod\limits_{n=0}^{N-1} \left( 1- \widehat{p}_n\right)\right)\widehat{P}^{\mbox{europ.}} \right|^m\right]\\
&\le \limsup \limits_{h \downarrow 0 } \mathbb{E}\left[| \widehat{P}^{\mbox{europ.}} |^m\right] < C_m.
\end{align*}
\end{proof}

\begin{lemma}\label{lemma:boundbarrier}
Provided assumptions A1 to A3 are satisfied, then the expected value of a up-and-out barrier call option \eqref{payoffcont} satisfies
\begin{align*}
\mathbb{E}[|P|^m]<C_m,
\end{align*}
for all positive integers $m$.
\end{lemma}
\begin{proof} Using \cref{thm:originalS} and the Lipschitz-continuity, with positive constant $L$, we obtain
\begin{align*}
\mathbb{E}[|P|^m]\le \mathbb{E}\left[ \left|\left(S(T)-K\right)^+\right|^m \right]\le L^m\mathbb{E} \left[|S(T)|^m \right]< C_m.
\end{align*}
\end{proof}

\begin{lemma}\label{giles3.11}
If $Y$ is a scalar random variable, $\mathbb{E}[Y^2]<\infty$ , and for each $p>0$, the indicator function $\mathbf{1}_{E}$ (which takes value $1$ or $0$, depending on whether or not a path lies within some set $E$), satisfies
\begin{align*}
\mathbb{E}[\mathbf{1}_{E}]=o(h^p),
\end{align*}
then for each $p>0$,
\begin{align*}
\mathbb{E}[|Y| \mathbf{1}_{E} ]= o(h^p).
\end{align*}
\end{lemma}
\begin{proof}
Immediate consequence of H{\"o}lder inequality.
\end{proof}

\begin{definition}\label{giles12.2}
We define the Brownian bridge continuous-time interpolant for $t_n\le t \le t_{n+1}$
\begin{align*}
\widehat{S}(t)&=\widehat{S}_{n}+\frac{(t-t_n)}{h}(\widehat{S}_{{n+1}}-\widehat{S}_{n})\\
&\quad+\sigma(\widehat{S}_{n},{t_n})\left(W(t)-\overline{W}(t)\right)
\end{align*}
with the piecewise linear interpolant $\overline{W}(t)=W_{n}- \frac{(t-t_n)}{h}(W_{{n+1}}-W_{n})$ of the discrete values of the Wiener process.
\end{definition}

\begin{lemma}\label{thm:extreme}
Provided assumptions A1 to A3 are satisfied, then for any $\gamma>0$, the probability that a Brownian path $W(t)$, its increments $\Delta W_n$, and the corresponding SDE solution $S(t)$ and its Brownian bridge continuous-time interpolant $\widehat{S}(t)$, satisfy any of the following extreme conditions
\begin{align*}
\max\limits_n \left( \max (\left|S(nh)\right|, |\widehat{S}_n|)\right) &> h^{-\gamma}\\
\max\limits_n \left(  \left| S(nh)-\widehat{S}_n\right| \right) &> h^{1-\gamma}\\
\max\limits_n \left|\Delta W_n\right| &> h^{1/2-\gamma}\\
\sup\limits_{[0,T]} \left| \widehat{S}(t) -S(t)\right| &> h^{1-\gamma}\\
\sup\limits_{[0,T]} \left|  W(t)-\overline{W}(t)     \right| &> h^{1/2-\gamma}
\end{align*}
is $o(h^p)$ for all $p>0$.

If none of these extreme conditions is satisfied, and $\gamma< 1/2$ then there exist positive constants $c_1,c_2$ and $c_3$ such that
\begin{align*}
\max\limits_n | \widehat{S}_n - \widehat{S}_{n-1}| &< c_1 h^{1/2-2\gamma}\\
\max\limits_n | \sigma(\widehat{S}_n,t_n) - \sigma(\widehat{S}_{n-1},t_{n-1}) | &< c_2 h^{1/2-2\gamma}\\
\max\limits_n \left( | \sigma(\widehat{S}_n,t_n)|   \right) &< c_3 h^{-\gamma}.
\end{align*}
\end{lemma}
\begin{proof}
See lemma 3.16 in \cite{giles2013numerical}.
\end{proof}

Before presenting the first convergence result, we will shortly explain the proof's proceeding, which is predetermined by \cref{thm:extreme}.
Studying the convergence, we especially are interested in the difference between $S$ and $\widehat{S}$ near the barrier. Crucial events may arise if, e.g., the maximum $S_{\max}$ of $S$ lays above the $B$, but $\widehat{S}$ stays below the barrier and vice versa. We will divide the paths into different subsets: First, we will see that paths that satisfy the \textit{extreme} conditions can be neglected since they are of order $o(h^p)$ for all $p>0$. That means, that, near the barrier, we will be able to focus on the following two non-\textit{extreme} cases : $|S_{\max} -B|\le h^{1/2-4\gamma}$ and $|S_{\max} -B|> h^{1/2-4\gamma}$, for $0<\gamma <\frac{1}{8}$. For these, we will examine the difference between $S$ and $\widehat{S}$ and, therefore, their contribution to $\mathbb{E}[P-\widehat{P}]$.

\begin{lemma}\label{convergence-slow}
Provided assumptions A1 to A5 are satisfied, then the Brownian bridge approach satisfies 
\begin{align*}
\mathbb{E}[P-\widehat{P}]  &< c h^{1/2-\delta},
\end{align*}
for any positive $c$ and $\delta$.
\end{lemma}

\begin{proof}
Let the maximum of $S(t)$ be $S_{\max}$. We divide the paths into the following three subsets: 
\begin{itemize}
\item[(i)] \textit{Extreme} conditions of \cref{thm:extreme} are satisfied, for $0<\gamma <\frac{1}{8}$.
\item[(ii)] \textit{Extreme} conditions of \cref{thm:extreme} are not satisfied, but $S_{\max}$ satisfies $|S_{\max} -B|> h^{1/2-4\gamma}$, for $0< \gamma < \frac{1}{8}$.
\item[(iii)] The rest
\end{itemize}
We have the following decomposition:
\begin{align*}
\mathbb{E}[P-\widehat{P}]=\mathbb{E}[(P-\widehat{P})\mathbf{1}_{(i)}]+\mathbb{E}[(P-\widehat{P})\mathbf{1}_{(ii)}]+\mathbb{E}[(P-\widehat{P})\mathbf{1}_{(iii)}],
\end{align*}
with the indicator functions to be the unit value for paths in the respected subset. For each subset we bound their contribution to $\mathbb{E}[P-\widehat{P}]$.
\begin{itemize}
\item[(i)] \cref{thm:oss} and \cref{lemma:boundbarrier} deliver bounds for $\mathbb{E}[|\widehat{P}|^2]$, $\mathbb{E}[|P|^2]$, and $\mathbb{E}[(P-\widehat{P})^2]$. Using \cref{thm:extreme}, we have that $\mathbb{E}[\mathbf{1}_{(i)}]$ is $o(h^p)$. Hence, using \cref{giles3.11}, we see that $\mathbb{E}[(P-\widehat{P})\mathbf{1}_{(i)}]$ is $o(h^p)$ for all $p>0$.
\item[(ii)]
Suppose $S(t)$ attains its maximum at $\tau\in [t_n,t_{n+1}]$.
First, consider the case $S_{\max}>B+h^{1/2-4\gamma}$. We study the Brownian bridge interpolant, since $S_{\max}$ could lay between two discretization steps. The first summand of the right-hand side of
\begin{align*}
| \widehat{S}_n-S_{\max} | \le | \widehat{S}_n- \widehat{S}(\tau) | + | \widehat{S}(\tau) -S(\tau) |
\end{align*}
can be written as
\begin{align*}
\widehat{S}(\tau)-\widehat{S}_n= \frac{\tau -t_n}{h}\left( \widehat{S}_{n+1}-\widehat{S}_n \right)+\sigma(\widehat{S}_{t_n},{t_n})\left(W(t)-\overline{W}(t)\right).
\end{align*}
From \cref{thm:extreme}, we obtain $\sup_{[0,T]} | \widehat{S}(t) -S(t)| < c h^{1-\gamma}$ and $\max_n | \widehat{S}_n - \widehat{S}_{n-1}| < c_1 h^{1/2-2\gamma}$. Together, we can conclude that $|\widehat{S}_n-S_{\max}|< c h^{1/2-2\gamma}$, for a positive constant $c$. Hence, for sufficiently small $h$ we have $|\widehat{S}_n-S_{\max}|< h^{1/2-4\gamma}$ and, therefore, $\widehat{S}$ is guaranteed to be greater than $B$ and hence $\widehat{P} -P=0$.
For the second case, we have $S_{\max}< B- h^{1/2-4\gamma}$. Here, we study the conditioning probabilities. We have
\begin{align*}
\max\limits_n \max (\widehat{S}_n) < B- h^{1/2-4\gamma} + h^{1-\gamma},
\end{align*}
since it is not extreme. Since $h^{1-\gamma} \prec h^{1/2-4\gamma}$ it follows that $\prod_{n} ( 1- \widehat{p}_n)$ is $1-o(h^p)$ for all $p>0$. Hence, with the Lipschitz-condition and the bound on $\widehat{S}_{N}-S(T)$ for non-\textit{extreme} paths, we obtain that $\mathbb{E}[(P-\widehat{P})\mathbf{1}_{(ii)}]$ is at most $O(h^{1-\gamma})$, since $\mathbb{E}[\mathbf{1}_{(ii)}]$ is $1$.
\item[(iii)] We have that $\mathbb{E}[\mathbf{1}_{(iii)}]$ is $O(h^{1/2-3\gamma})$ due to the bounded density of $S$ in the neighbourhood of $B$. Furthermore, $\mathbb{E}[|\widehat{P}|^2]$ and $\mathbb{E}[|P|^2]$ are bounded. Together, we obtain that $\mathbb{E}[(P-\widehat{P})\mathbf{1}_{(iii)}]$ is at most $O(h^{1/2-3\gamma})$. 
\end{itemize}
Finally, the proof is completed by choosing $\gamma < \min (\frac{1}{8},\delta /3)$.
\end{proof}
We see that the third case is decisively for the inferior convergence order. Nevertheless, we will overcome this issue through the following lemma. The proof is a modification of the proof of theorem 3.16 of \cite{giles2013numerical}, where the authors prove $\mathbb{E}[(\widehat{P}_l-\widehat{P}_{l-1})^2]$ is of $o(h^{3/2-\gamma})$.

\begin{lemma}\label{cauchy}
Provided assumptions A1 to A5 are satisfied, then the Brownian bridge approach forms a Cauchy series. By using $h_l=2^{-l}$, with $l=0,1,\dots$, the series satisfies
\begin{align*}
\mathbb{E}[\widehat{P}_l-\widehat{P}_{l-1}]  &< c h_l^{1-\delta},
\end{align*}
for any positive $c$ and $\delta$.
\end{lemma}
\begin{proof}
See the proof of theorem 3.16 of \cite{giles2013numerical}. Aiming to bound $\mathbb{E}[(\widehat{P}_l-\widehat{P}_{l-1})^2]$, the authors prove the following results for the three subsets and for $0<\gamma<\frac{1}{8}$ (see the subsets of \cref{convergence-slow}):
\begin{itemize}
\item[(i)]  $\mathbb{E}[\mathbf{1}_{(i)}]$ is $o(h^p)$.
\item[(ii)] There exists a constant $c>$ such that $\widehat{P}_l-\widehat{P}_{l-1}<ch^{1-\gamma}$.
\item[(iii)] There exists a constant $c>$ such that $\widehat{P}_l-\widehat{P}_{l-1}<ch^{1/2-6\gamma}$.
\end{itemize}
Now, using slight modifications, aiming to bound $\mathbb{E}[\widehat{P}_l-\widehat{P}_{l-1}]$ instead of $\mathbb{E}[(\widehat{P}_l-\widehat{P}_{l-1})^2]$, we obtain:
\begin{itemize}
\item[(i)]  $\mathbb{E}[(\widehat{P}_l-\widehat{P}_{l-1})\mathbf{1}_{(i)}]$ is $o(h^p)$ for all $p>0$.
\item[(ii)]  $\mathbb{E}[(\widehat{P}_l-\widehat{P}_{l-1})\mathbf{1}_{(ii)}]$ is $O(h^{1-\gamma})$.
\item[(iii)] There exists a constant $c>$ such that $\widehat{P}_l-\widehat{P}_{l-1}<ch^{1/2-6\gamma}$.
We have that $\mathbb{E}[\mathbf{1}_{(iii)}]$ is $O(h^{1/2-3\gamma})$ due to the bounded density of $S$ in the neighbourhood of $B$. Together with $\widehat{P}_l-\widehat{P}_{l-1}<ch^{1/2-6\gamma}$, we obtain that $\mathbb{E}[(\widehat{P}_l-\widehat{P}_{l-1})\mathbf{1}_{(iii)}]$ is at most $O(h^{1-9\gamma})$.
\end{itemize}
Finally, the proof is completed by choosing $\gamma < \min (\frac{1}{8},\delta /9)$.
\end{proof}

Finally, we combine the results of \cref{convergence-slow} and \cref{cauchy} to formulate the proof of \cref{thmconvergence}.
\begin{proof} [Proof of \cref{thmconvergence}]\label{thm:316}
The weak convergence result of \cref{convergence-slow} together with the result on the Cauchy series of \cref{cauchy}, complete the first part of the proof.
Moreover, by applying \cref{thm:oss}, we obtain that there exits a positive constant $c$ such that
\begin{align*}
\limsup \limits_{h \downarrow 0 } \mathbb{E}\left[ |\widehat{P}|^2 \right] < c.
\end{align*}
Thus, we obtain the variance bound
\begin{align*}
\limsup\limits_{h \downarrow 0 } \mbox{Var}[\widehat{P}]= \limsup\limits_{h \downarrow 0 } (\mathbb{E}[\widehat{P}^2]- \mathbb{E}[\widehat{P}]^2)\le \limsup\limits_{h \downarrow 0 } \mathbb{E}[\widehat{P}^2] < c,
\end{align*}
which completes the proof.
\end{proof}

\end{document}